\documentclass[alphabetic]{amsart}

\usepackage{enumerate, longtable}
\usepackage{amsmath, amscd, amsfonts, amsthm, amssymb, latexsym, comment, stmaryrd, graphicx, dsfont, amsaddr}
\usepackage{xcolor}
\usepackage[all]{xy}
\usepackage{fullpage}
\usepackage{tikz}
\usepackage[
        colorlinks,
        linkcolor=red,  citecolor=blue,
        backref
]{hyperref}

\usepackage{amstext} 
\usepackage{array}
\usepackage{amsrefs}

\newtheorem{thm}{Theorem}
\newtheorem{lem}{Lemma}[section]

\newtheorem{claim}[lem]{Claim}

\newtheorem{prop}[lem]{Proposition}
\newtheorem{conj}[lem]{Conjecture}
\newtheorem{question}[lem]{Question}
\newtheorem{observation}[lem]{Observation}

\theoremstyle{definition}

\newtheorem{definition}[lem]{Definition}
\newtheorem{rem}[lem]{Remark}

\numberwithin{equation}{section}

\newcommand{\Gal}{\mathrm{Gal}}
\newcommand{\Disc}{\mathrm{Disc}}
\newcommand{\Res}{\mathrm{Res}}
\newcommand{\F}{\mathbb{F}}
\newcommand{\Fb}{{\overline{\F}}}
\newcommand{\Z}{\mathbb{Z}}
\newcommand{\chr}{\mathrm{char}\,}
\newcommand{\ab}{{\mathrm{ab}}}

\newcommand{\field}[1]{\mathbb{F}_{#1}}
\newcommand{\fourierco}[1]{\widehat{{#1}}}
\newcommand{\rad}{\mathrm{rad}}
\newcommand{\LegendreP}[2]{\left(\frac{{#1}}{{#2}}\right)}

\newcommand{\intval}[1]{\left\lfloor{#1}\right\rfloor}
\newcommand{\abs}[1]{\left| {#1}\right|}

\newcommand{\lgtq}{{L^{\mathrm{gt}}_1(q)}}
\newcommand{\lgtp}{{L^{\mathrm{gt}}_1(p)}}

\makeatletter
\@namedef{subjclassname@2020}{%
  \textup{2020} Mathematics Subject Classification}
\makeatother

\begin{document}

\title{Abhyankar's Affine Arithmetic Conjecture for the Symmetric and Alternating Groups}
\author{Alexei Entin}
\address{Raymond and Beverly Sackler School of Mathematical Sciences, Tel Aviv University,
Tel Aviv 69978, Israel}
\email{aentin@tauex.tau.ac.il (corresponding author)}
\author{Noam Pirani}
\address{Raymond and Beverly Sackler School of Mathematical Sciences, Tel Aviv University,
Tel Aviv 69978, Israel}
\email{noampirani@mail.tau.ac.il}


\maketitle

\begin{abstract} 
We prove that for any prime $p>2$, $q=p^\nu$ a power of $p$, $n\ge p$ and $G=S_n$ or $G=A_n$ 
(symmetric or alternating group) there exists a Galois extension $K/\F_q(T)$ ramified only over $
\infty$ with $\Gal(K/\F_q(T))=G$. This confirms a conjecture of Abhyankar for the case of symmetric and alternating 
groups over finite fields of odd characteristic.
\end{abstract}

\section{Introduction} Let $k$ be an algebraically closed field and $k(T)$ the univariate rational function field over $k$. It is a well-known and elementary fact that if $\chr k=0$, any finite extension $K/k(T)$ which is ramified over at most a single place is trivial. Equivalently, we have $\pi_1\left(\mathbb{A}^1_k\right)=1$, where $\pi_1$ denotes the \'etale fundamental group and $\mathbb{A}^1_k$ is the affine line over $k$. This reflects the topological fact that the complex plane $\mathbb C$ is simply connected. 
The analogous question for $\chr k=p>0$ is much more interesting. For a finite group $G$ and prime $p$ denote by $p(G)$ the group generated by the $p$-Sylow subgroups of $G$. We say that $G$ is \emph{quasi-$p$} if $G=p(G)$. It is not hard to show using elementary ramification theory that if $K/k(T)$ is finite Galois with $\Gal(K/k(T))=G$ and is ramified over a single point of $\mathbb P^1_k$ (WLOG we may assume it is the point at infinity) then $G$ must be quasi-$p$. In 1957 Abhyankar conjectured that the converse holds.

\begin{conj}[Abhyankar's Conjecture for the affine line]\label{conj:ac} Let $k$ be algebraically closed of characteristic $\chr k=p>0$ and $G=p(G)$ a finite quasi-$p$ group. Then there exists a Galois extension $K/k(T)$ ramified only over $\infty$ such that $\Gal(K/k(T))=G$.\end{conj}

The above conjecture is the prototype special case of the more general \cite{Abh57}*{Conjecture 1}. Conjecture \ref{conj:ac} was settled in a groundbreaking work of Raynaud \cite{Ray94}, building on previous work by Serre \cite{Ser90}. It was extended to the case of a general affine curve by Harbater \cite{Har94}.

Starting in 1992 and throughout the 90's, Abhyankar (with several collaborators) undertook the project of producing explicit equations realizing various "nice" groups as their Galois groups over $\overline \F_p(T)$ with ramification over a single geometric point (or two geometric points). This resulted in a series of papers collectively titled \emph{Nice equations for nice groups}, which covered the symmetric and alternating groups \cite{Abh92, Abh93, AOS94, AbYi94}, groups of the form $\mathcal G(\F_{p^k})$ where $\mathcal G$ is a classical group and $p$ is the characteristic \cite{Abh94,Abh96,Abh96a,Abh96b,AbLo98,AbLo99,Abh99,AbIn01} and some Mathieu groups in characteristic 2,3 \cite{Abh95}. The equations obtained by Abhyankar et al. were defined over finite fields and these investigations led to the statement in \cite{Abh01}*{\S 16} of an arithmetic version of Conjecture \ref{conj:ac}, which is the focus of the present paper.

A finite group $G$ is called \emph{cyclic-by-quasi-$p$} if it is an extension of a cyclic group by a 
quasi-$p$ group. Equivalently, $G/p(G)$ is cyclic. If $q$ is a power of $p$ and $K/\F_q(T)$ is a 
finite Galois extension with $\Gal(K/\F_q(T))=G$ ramified only over $\infty$, then $G$ must be cyclic-by-quasi-$p$ because $
\Gal(K\overline\F_q/\overline \F_q(T))$ is quasi-$p$ and\\ $\Gal(\overline \F_q(T)/\F_q(T))=\Gal(\overline\F_q/\F_q)$ is cyclic.

\begin{conj}\label{aac}[Abhyankar's Arithmetic Conjecture for the affine line] Let $q$ be a power of a prime $p$ and $G$ a finite cyclic-by-quasi-$p$ group. Then there exists a Galois extension $K/\F_q(T)$ ramified only over $\infty$ such that $\Gal(K/\F_q(T))=G$.\end{conj}

Equivalently, the finite quotients of $\pi_1\left(\mathbb A^1_{\F_q}\right)$ are exactly the cyclic-by-quasi-$p$ groups. For a good survey on Abhyankar's Arithmetic Conjecture and related conjectures 
and partial results see \cite{HOPS17}*{\S 5}. The series of works by Abhyankar et al. cited above 
proves many instances of the conjecture for classical groups of the form $\mathcal G(\F_{p^k})$ ($p=
\chr\F_q$) and also many instances for the symmetric and alternating groups $S_n,A_n,n\ge p$ (for 
$n>2$ the condition $n\ge p$ is equivalent to $S_n$ being cyclic-by-quasi-$p$ and for $n>3$ the same is true for $A_n$), 
however most of their constructions have Galois group either $A_n$ or $S_n$ depending on the value 
of a certain Legendre symbol involving $n,q$ and possibly other parameters. Consequently, for many 
pairs $(n,q)$ Conjecture \ref{aac} had been previously proved for only one of $A_n,S_n$ and it is 
the main goal of the present work to cover both $S_n,A_n$ for all $n\ge p>2$. Our main result is the 
following

\begin{thm}\label{thm:main} Let $q$ be a power of a prime $p>2$, $n\ge p$. Then Conjecture \ref{aac} holds with $G=S_n$ and $G=A_n$.
\end{thm}

The case of $G=A_n,n\neq p+1$ (and also $n=p+1$ if $\F_q\supset\F_{p^2}$) was recently proved in \cite{BEF21_}*{Theorem 1.10}). To treat the case $G=S_n$ we show how to modify certain $S_n$-extensions $K/\F_q(T)$ ramified over two points (of $\mathbb P^1(\F_q)$) constructed in \cite{BEF21_} into $S_n$-extensions with just one ramified point, using Abhyankar's method of removing tame ramification (MRT). This will be done in section \ref{sec:lgtq}, where we also give the proof of Theorem \ref{thm:main} except for one key step required for the case $n=p+1$, which we defer to the subsequent sections.

To treat the case $n=p+1$ we make use of a new construction based on the polynomial  \begin{equation}\label{eq:new pol}f=X^p(X-1)-T(X-4)^{p-1}\left(X-\frac 43\right)\in\F_q[T,X],\end{equation} which was suggested to us by the anonymous referee of a previous version of this paper. Before learning of this more convenient construction, we used a construction due to Abhyankar \cite{Abh92}*{\S 12}, based on polynomials of the form
\begin{equation}\label{eq:abh poly} f=(X+1)\left(X+\frac a{a-1}\right)^p-T^{-a(p+1-a)}X^a\in \F_q[T,X].\end{equation}
 The latter approach is more complicated and excludes the case $G=S_{p+1},\F_q\supset\F_{p^2}$. Nevertheless we decided to include it here because having an alternative construction may prove useful in some future applications (e.g. if a realization is sought with some additional requirements) and its analysis involves a number-theoretic method that may be useful in other contexts (e.g. in studying the tame version of the problems studied here). The approach through the polynomial (\ref{eq:abh poly}) 
requires choosing the parameter $2\le a\le\frac{p-1}2$ so that $(a,p+1)=1$ and then the resulting Galois 
group is $A_n$ or $S_n$ depending on whether $\LegendreP {a(a-1)}p=1$ or $-1$. The existence of a suitable parameter $a$ is guaranteed by the following purely number-theoretic statement.

\begin{prop}\label{prop:legendre}
Let $p>13$ be a prime, $\sigma=\pm 1$. There exists an integer $2\le a\le \frac{p-1}{2}$, such that $(a,p+1)=1$ and $\LegendreP{a(a-1)}{p}=\sigma$.
\end{prop}

This is easy to prove for $p$ sufficiently large using P\'olya-Vinogradov-type estimates, but doing it for all $p>13$ takes more effort. To achieve this we combine the P\'olya-Vinogradov 
method with an elementary sieve argument to cover the range $p>7\cdot 10^7$ and a computer search to 
deal with $p<7\cdot 10^7$. This is carried out in section \ref{sec:legendre}. Our 
approach here is similar in spirit to (and inspired by) the series of works by S. D. Cohen et al. on 
the existence of primitive elements in finite fields satisfying additional properties (see e.g. 
\cite{Coh85, CoHu03, CoHu03a, CoHu10}). The derivation of the case $n=p+1$ of Theorem \ref{thm:main} from Proposition \ref{prop:legendre} is carried out in section \ref{sec:np1}.

\begin{rem} The present paper focuses on the case $p>2$. For $p=2$, Conjecture \ref{aac} holds for $S_n$. For $A_n$, Conjecture \ref{aac} holds if $\F_q\supset\F_4$ or if $10\neq n\ge 8$ and $n\equiv 0,1,2,6,7\pmod 8$. This can be seen from the proof of \cite{BEF21_}*{Theorem 5.7}, or derived from its statement using Proposition \ref{prop:lgtq_lq} below, similarly to how we derive the case $n\neq p+1,p>2$ in section \ref{sec:lgtq}. We note that the $p=2$ case of \cite{BEF21_}*{Theorem 5.7} is based directly on constructions by Abhyankar, Ou, Sathaye and Yie appearing in \cite{Abh92,AOS94,AbYi94}. The problem for $A_n$ remains open for most other values of $n$ with $p=2$.\end{rem} 

\begin{rem} The proof of Theorem \ref{thm:main} for the case $n=p+2$ uses the classification of finite simple groups (CFSG). All other cases only use 19th century group theory. See Remark \ref{rem:cfsg} for more details.\end{rem}

\subsection{Application to the minimal ramification problem}

As a corollary to the case $G=A_{p+1}$ of Theorem \ref{thm:main} we obtain an improvement to a 
result from \cite{BEF21_} on the minimal ramification problem. Instead of looking for $G$-extensions 
$K/\F_q(T)$ ramified over a single prime divisor of degree 1 as in Conjecture \ref{aac}, one can relax 
the ramification condition to a single prime divisor of any degree, or more generally ask what is 
the minimal number of ramified prime divisors of $\F_q(T)$ occurring in a $G$-extension of $\F_q(T)$. To state the general 
conjecture on this problem we denote by $d(G)$ the minimal number of generators of a finite group $G
$ and by $G^{\ab}$ the abelianization of $G$. We call an extension $L/K$ of univariate function 
fields \emph{geometric} if $K$ and $L$ have the same field of constants. The following conjecture 
generalizes the corresponding conjecture by Boston and Markin over $\mathbb Q$ \cite{BoMa09}*{Conjecture 1.2}.

\begin{conj}[\cite{BEF21_}*{Conjecture 1.4}]\label{conj:minram} Let $G$ be a nontrivial finite group. Then there exists a geometric Galois extension $K/\F_q(T)$ with $\Gal(K/\F_q)=G$ and ramified over at most $r$ prime divisors of $\F_q(T)$ iff $r\ge d\left(G^\ab/p(G^\ab)\right)$.\end{conj}

For a more thorough discussion of Conjecture \ref{conj:minram}, the original Boston-Markin conjecture and known partial results see \cite{BEF21_}*{\S 1}. In the case $G=S_n,A_n$ Conjecture \ref{conj:minram} predicts that there exists a geometric $G$-extension $K/\F_q(T)$ ramified over a single prime divisor, for all $n,q$. Significant progress on the $G=A_n,S_n$ case of the conjecture was made in \cite{BEF21_}, but for many values of $n,q$ it remains open. Using one of our auxiliary results, namely Proposition \ref{prop:sp1 realization}, we obtain the following strengthening of \cite{BEF21_}*{Theorem 1.8} in the case of odd characteristic:

\begin{thm}\label{thm:minram} Assume $p>2$, $q$ a power of $p$. Let $G=G_1\times\cdots\times G_m$ be a product with each $G_i$ isomorphic to $S_{n_i}$ or $A_{n_i}$ with $n_i\ge p$.
Then Conjecture \ref{conj:minram} holds for $G$ over $\F_q$.
\end{thm}

In the original statement of \cite{BEF21_}*{Theorem 1.8} the case $G_i=S_{p+1},A_{p+1}$ was excluded. The derivation of Theorem \ref{thm:minram} will be given in section \ref{sec:np1}.

\section{$\lgtq$-realizations and removing tame ramification}\label{sec:lgtq}

We are assuming the reader's familiarity with the basic theory of function fields in one variable over finite fields. Good introductory sources covering the necessary background are \cite{Ros02} and \cite{Sti09}. Throughout the rest of the paper $p$ denotes a prime number and $q$ is a power of $p$. We denote the cyclic group of order $n$ by $C_n$. The symbol $\subset$ denotes non-strict inclusion.

\begin{definition}\label{def:lq} A finite group $G$ is called \emph{$L(q)$-realizable} if there exists a Galois extension $K/\F_q(T)$ with $\Gal(K/\F_q(T))\cong G$ which is unramified outside of a single prime divisor of $\F_q(T)$. An extension with this property is called an \emph{$L(q)$-realization} of $G$.\end{definition}

Note that in Definition \ref{def:lq} we could require WLOG that the ramified point is $\infty$, since we can always move a given point of $\mathbb P^1(\F_q)$ to $\infty$ by an automorphism of $\mathbb P^1_{\F_q}$ (equivalently, by a linear fractional transformation of the variable $T$ with coefficients in $\F_q$). Hence Conjecture \ref{aac} is equivalent to the assertion that $G$ is $L(q)$-realizable iff it is cyclic-by-quasi-$p$. Next we introduce a closely related definition taken from \cite{BEF21_}*{Definition 5.4}. 


\begin{definition}\label{def:lgtq} A finite group $G$ is called \emph{$\lgtq$-realizable} if there exists a geometric Galois extension $K/\F_q(t)$ with $\Gal(K/\F_q(t))\cong G$, $K/\F_q(T)$ is unramified outside of two degree 1 prime divisors of $\F_q(T)$ and at most tamely ramified over one of them. An extension with this property is called an \emph{$\lgtq$-realization} of $G$.\end{definition}

Once again we could require WLOG that the ramified points are $0,\infty$ and the tame ramification is over (say) $\infty$.

\begin{rem} The reason for the notation is Abhyankar's notation $L,L_1$ for the affine line and once-punctured affine line respectively. The letters "gt" stand for geometric and tame. \end{rem}

In the present section we show how to modify $\lgtq$-realizations into $L(q)$-realizations (using Abhyankar's method of removing tame ramification by cyclic compositum, or MRT), which combined with the existence of $\lgtq$-realizations constructed in \cite{BEF21_}*{\S 5} will allow us to prove the $n\neq p+1$ cases of Theorem \ref{thm:main}. First we will need the following

\begin{lem}\label{lem:lgtq} Let $G$ be $\lgtq$-realizable. Then
\begin{enumerate}
\item[(i)] $G$ is cyclic-by-quasi-$p$ and $e=[G:p(G)]$ divides $q-1$.
\item[(ii)] There exists an $\lgtq$-realization $K/\F_q(T)$ with $\Gal(K/\F_q(T))=G$ and $K^{p(G)}=\F_q(S)$ (identifying $G$ with $\Gal(K/\F_q(T))$), where $S^e=T$. The extension $K/\F_q(S)$ is unramified outside the divisor $S=0$.
\end{enumerate}
\end{lem}

\begin{proof} This is a slightly more precise statement of \cite{BEF21_}*{Lemma 5.5} and the proof in \cite{BEF21_} in fact shows this much.\end{proof}

In light of Lemma \ref{lem:lgtq}(i) and by analogy with Conjecture \ref{aac}, it is natural to raise the following

\begin{question}[\cite{BEF21_}*{Question 5.6}] Is a cyclic-by-quasi-$p$ group $G$ with $[G:p(G)]$ dividing $q-1$ always $\lgtq$-realizable?\end{question}

The next proposition will imply that a positive answer to the above question for a given group $G$ satisfying its assumptions implies that Conjecture \ref{aac} holds for $G$ (over $\F_q$).

\begin{prop}[modifying an $\lgtq$-realization into an $L(q)$-realization using MRT] \label{prop:lgtq_lq} Let $G$ be an $\lgtq$-realizable group. Then $G$ is $L(q)$-realizable.\end{prop}

\begin{proof} By Lemma \ref{lem:lgtq}, there exists an extension $K/\F_q(T)$ which is geometric, has 
Galois group $\Gal(K/\F_q(T))=G$, with $K^{p(G)}=\F_q(S)$, where $S^e=T$, $e=[G:p(G)]\mid q-1$ and $K/\F_q(S)$ is unramified outside of $S=0$. Let 
$\zeta\in\F_{q^e}$ be a primitive $e(q-1)$-th root of unity. Then $\F_q(\zeta)=\F_{q^e}$ and $
\F_{q^e}(S)=\F_{q^e}(\zeta S)$ is a $C_e$-extension of both $\F_q(S)$ and $\F_q(\zeta S)$, which are 
in turn $C_e$-extensions of $\F_q(T)$ (note that $(\zeta S)^e=\zeta^eT$ and $\zeta^e\in\F_q$).

Since $K/\F_q(T)$ is geometric, $K/\F_q(S)$ is also geometric ($\F_q(S)\supset\F_q(T)$) and is therefore linearly disjoint from $\F_{q^e}(S)/\F_q(S)$. We obtain the following diagram, with each marked extension being Galois with Galois group isomorphic to the indicated group.

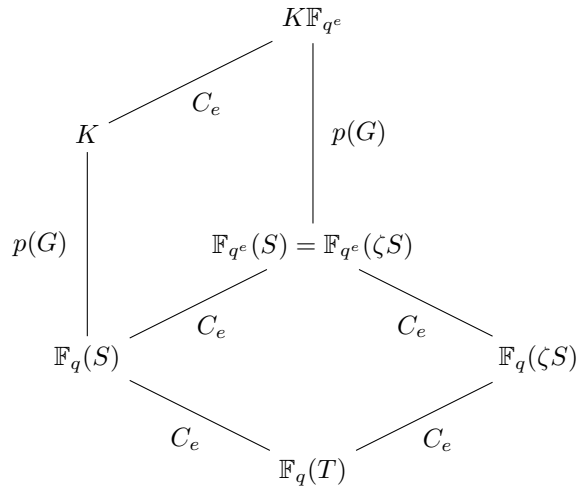
\begin{center}
\begin{tikzpicture}

    \node (Q1) at (0,0) {$\F_q(T)$};
    \node (Q2) at (-3,1.5) {$\F_q(S)$};
    \node (Q3) at (-3,4.5) {$K$};
    \node (Q4) at (0,6) {$K\F_{q^e}$};
    \node (Q5) at (0,3) {$\F_{q^e}(S)=\F_{q^e}(\zeta S)$};
    \node (Q6) at (3,1.5) {$\F_q(\zeta S)$};

    \draw (Q1)--(Q2) node [pos=0.6, below,inner sep=0.25cm] {$C_e$};
    \draw (Q2)--(Q3) node [pos=0.5, left,inner sep=0.25cm] {$p(G)$};
    \draw (Q3)--(Q4) node [pos=0.6, below,inner sep=0.25cm] {$C_e$};
    \draw (Q4)--(Q5) node [pos=0.5, right,inner sep=0.25cm] {$p(G)$};
    \draw (Q6)--(Q5) node [pos=0.6, below,inner sep=0.25cm] {$C_e$};
    \draw (Q1)--(Q6) node [pos=0.6, below,inner sep=0.25cm] {$C_e$};
    \draw (Q2)--(Q5) node [pos=0.6, below,inner sep=0.25cm] {$C_e$};

\end{tikzpicture}
\end{center}

Since $K\F_{q^e}=K\F_{q^e}(\zeta S)$ ($S\in K$ and $\F_q(\zeta)=\F_{q^e}$) we have $[K\F_q(\zeta S):K]=[\F_q(\zeta S):\F_q(T)]$ and consequently $K/\F_q(T)$ and $\F_q(\zeta S)/\F_q(T)$ are linearly disjoint. Consequently $K\F_{q^e}/\F_q(\zeta S)$ is Galois with $\Gal(K\F_{q^e}/\F_q(\zeta S))\cong\Gal(K/\F_q(T))=G$. 

Denoting $U=\zeta S$, we claim that $K\F_{q^e}/\F_q(U)$ is the sought $L(q)$-realization of 
$G$. First note that $F_q(S)/\F_q(T)$ is totally and tamely ramified over $T=0,\infty$ (since $
(e,p)=1$) and unramified elsewhere, and recall that $K/\F_q(S)$ is ramified only over $S=0$. Consequently $K
\F_{q^e}/\F_{q^e}(S)$ is also ramified only over $S=0$. Since $\F_{q^e}(S)/\F_q(U)$ is a constant 
field extension, it is unramified everywhere and only $S=0$ lies over $U=0$. Consequently $K
\F_{q^e}/\F_q(U)$ is unramified outside $U=0$ and is therefore an $L(q)$-realization of $G$.
\end{proof}

\begin{prop}\label{prop:lgtq}Let $n\ge p>2$. Then $S_n,A_n$ are $\lgtq$-realizatble.\end{prop}

\begin{proof} We are assuming that $n\ge p>2$. If $n\neq p+1$ then by \cite{BEF21_}*{Theorem 5.7} the groups $S_n,A_n$ are $\lgtq$-realizable. If $n=p+1$ we will construct an $\lgtq$-realization $K/\F_q(t)$ of $G=S_{p+1}$ in Proposition \ref{prop:sp1 realization} below, and by Lemma \ref{lem:lgtq}(ii) $K^{p(G)}/\F_q(t)$ is an $\lgtq$-realization of $p(G)=p(S_{p+1})=A_{p+1}$. \end{proof}

\begin{proof}[Proof of Theorem \ref{thm:main}] Assume $n\ge p>2$. By the last proposition the groups $S_n,A_n$ are $\lgtq$-realizable. We have $p(A_n)=p(S_n)=A_n$ and hence $[G:p(G)]\in\{1,2\}$ divides $q-1$ if $G=S_n,A_n$. By Proposition \ref{prop:lgtq_lq}, $G$ is $L(q)$-realizable, i.e. the assertion of Conjecture \ref{aac} holds.\end{proof}

\begin{proof}[Proof of Theorem \ref{thm:minram}] The original proof of \cite{BEF21_}*{Theorem 1.8} goes through verbatim once one shows that each $G_i$ is $\lgtq$-realizable. This is provided by Proposition \ref{prop:lgtq}.\end{proof}

\begin{rem}\label{rem:cfsg} The proof of \cite{BEF21_}*{Theorem 5.7}, which we used in the proof of Theorem \ref{thm:main} for $n\neq p+1$, relies on the classification of finite simple groups (CFSG) in the case $n=p+2$. It is used to deduce that the Galois group contains $A_n$ from primitivity and the existence of elements with certain cycle structures, via a theorem of Jones \cite{Jon14}*{Corollary 1.3} (which relies on the CFSG). Consequently the case $n=p+2$ of Theorem \ref{thm:main} also relies on CFSG. All other cases of \cite{BEF21_}*{Theorem 5.7} only use elementary group theory, mainly Jordan's theorem \cite{DiMo96}*{Theorem 3.3E}. The proof of \cite{Abh92}*{\S 12.IV.3} (unlike many other similar results in the cited source), which we used in the proof of Proposition \ref{prop:mtr} (i.e. the case $n=p+1$), also involves only elementary group theory, mainly Marggraf's theorem \cite{Wie64}*{Theorem 13.5}. Thus the $n\neq p+2$ cases of Theorems \ref{thm:main} and \ref{thm:minram} do not use the CFSG.\end{rem}

\section{The case $n=p+1$: first construction}
\label{sec:np1 new}

\begin{prop}\label{prop:sp1 realization}\begin{enumerate}\item[(i)] Assume $p>3$ and consider the polynomial
\begin{equation}\label{eq:new pol 1}f=X^p(X-1)-t(X-4)^{p-2}\left(X-\frac 43\right)\in\F_q[T,X].\end{equation}
Then the splitting field $K/\F_q(t)$ of $f$ is an $\lgtq$-realization of $S_{p+1}$ and in particular $S_{p+1}$ is $\lgtq$-realizable.
\item[(ii)] If $p=3$ then the splitting field of $f=X^4+X+T$ in an $\lgtq$-realization of $S_4$.\end{enumerate}\end{prop}

\begin{proof}{\bf (i).}To find the ramification locus we compute the discriminant of $f$. First we compute the derivative: 
$$f'=X^p+2T(X-4)^{p-3}\left(X-\frac43\right)^2-2T(X-4)^{p-2}\left(X-\frac 43\right)
=X^p+\frac{16}3(X-4)^{p-3}\left(X-\frac 43\right)T.$$
Hence
$$f+\frac 3{16}(X-4)\left(X-\frac 43\right)f'=X^p(X-1)+\frac 3{16}(X-4)\left(X-\frac 43\right)X^p = \frac 3{16}X^{p+2}$$
and
\begin{multline*}\Disc(f)=(-1)^{\frac{p+1}2}\Res(f,f')=(-1)^{\frac{p-1}2}\Res\left(f+\frac 3{16}(X-4)\left(X-\frac 43\right)f',f'\right)\\=(-1)^{\frac{p-1}2}\Res\left(\frac {3}{16}X^{p+2},f'\right)=(-1)^{\frac{p-1}2}\frac {3}{16}f'(0)^{p+2}=(-1)^{\frac{p+1}2}\frac{4}{243}T^{p+2}
\end{multline*}
(for a summary of the basic identities we used for manipulating the resultants and discriminants above see \cite{BEF21_}*{\S 2.4}). From the above it follows that $K/\F_q(T)$ is ramified only over 0 and possibly $\infty$.

Let $\alpha\in K$ be a root of $f$ and note that by (\ref{eq:new pol 1}) we have $$t=\frac{\alpha^p(\alpha-1)}{(\alpha-4)^{p-2}(\alpha-4/3)^2}$$ and hence $L=\F_q(T)(\alpha)=\F_q(\alpha)$ and the ramification degrees of the primes of $L/\F_q(T)$ lying over $0,\infty$ are $(p,1,1)$ and $(p-2,2,1)$ respectively. In particular we see that ramification over $\infty$ is tame. From this it follows (by e.g. \cite{Abh92}*{Cycle Lemma}) that $\Gal(K\Fb_q/\Fb_q(T))=\Gal(f/\Fb_q(T))\leqslant S_{p+1}$ contains both a transposition and a $p$-cycle. The group $\Gal(f/\Fb_q(T))$ is transitive (since $f$ is irreducible being linear in $t$), hence $\Gal(K\Fb_q/\F_q(T))=\Gal(f/\Fb_q(T))=S_{p+1}$ and since $\Gal(K\Fb_q/\Fb_q(T))\leqslant\Gal(K/\F_q(T))\leqslant S_{p+1}$ the extension $K/\F_q(T)$ is geometric with Galois group $S_{p+1}$ and we obtained an $\lgtq$-realization of $S_{p+1}$.

{\bf (ii).} A simple calculation shows $\Disc(f)=T^3$ and then reasoning as in part (i) we see that all ramification is over $0,\infty$, ramification over $\infty$ is tame, $\Gal(f/\Fb_q(T))\leqslant S_4$ is transitive, contains a 3-cycle and a 4-cycle, hence is all of $S_4$. From this it follows as in part (i) that the splitting field $K/\F_q(T)$ of $f$ is an $\lgtq$-realization of $S_4$.
\end{proof}

\section{The case $n=p+1$: second construction}\label{sec:np1}

In the next two sections we give an alternative treatment of the case $n=p+1$ (of Theorem \ref{thm:main}), which is based on a construction from \cite{Abh92}, summarized in the following

\begin{prop}\label{prop:mtr} Assume $p>5$ and let $2\le a\le\frac{p-1}2$ be an integer such that $(a,p+1)=1$. Consider the polynomial
$$f=(X+1)\left(X+\frac a{a-1}\right)^p-T^{-a(p+1-a)}X^a\in \F_q(T)[X]$$
and let $K$ be the splitting field of $f$ over $\F_q(T)$. Then the extension $K/\F_q(T)$ is ramified only over $\infty$ and
$$\Gal(f/\F_q(T))=\Gal(K/\F_q(T))=\left[\begin{array}{ll}A_n,&\LegendreP {a(a-1)}p=1\mbox{ or }\F_q\supset\F_{p^2},\\ \quad \\S_n,&\LegendreP {a(a-1)}p=-1\mbox{ and }\F_q\not\supset\F_{p^2}.\end{array}\right.$$

\end{prop}

\begin{proof} By \cite{Abh92}*{\S 12.IV.3} (taken with $t = a, b = \frac a{a-1}, s = a(p+1-a)$ in the notation of the cited source), under our assumptions the extension $K/\F_q(T)$ is ramified only over $
\infty$ and $\Gal(f/\overline \F_q(T))=A_{p+1}$. Note that $\deg f=p+1$ and hence $A_{p+1}\leqslant
\Gal(f/\overline \F_q(T))\leqslant\Gal(f/\F_q(T))\leqslant S_{p+1}$ and $G=\Gal(f/\F_q(T))$ equals 
$S_{p+1}$ or $A_{p+1}$ depending on whether $G\subset A_{p+1}$, which by the discriminant criterion 
happens iff $\Disc(f)\in\F_q(T)$ is a square.

By \cite{Abh92}*{\S 22} (taken with $\tau = a, Y=T^{-a(p+1-a)}, b=\frac a{a-1}$ in the 
notation of the cited source) we have
\begin{equation}\label{eq:disc_mtr}\Disc(f)=(-1)^{(p+1)/2}\frac{a^{pa-p+a}}{(a-1)^{pa-2p+a-1}}T^{-a(p+1-a)(p+1)}\end{equation}
(note that \cite{Abh92} uses a nonstandard sign convention for the discriminant, while we are using the standard definition of the discriminant as it appears in e.g. \cite{BEF21_}*{\S 2.4}, for which the discriminant criterion applies; in \cite{Abh92} the standard discriminant is called the modified discriminant and denoted by $\Disc^*$).

Since $T^{a(p+1-a)(p+1)}$ is always a square ($p$ is odd), $\Disc(f)$ is a square iff $u=(-1)^{(p+1)/2}a(a-1)$ is a square in $\F_q$ (since $(a,p+1)=1$, $a$ is odd and so are the exponents of $a$ and $a-1$ in (\ref{eq:disc_mtr})), which happens iff $u$ is a square in $\F_p$ or $\F_q\supset\F_{p^2}$. Finally, 
$u$ is a square in $\F_p$ iff $\LegendreP {a(a-1)}p=1$ (note that $(-1)^{(p+1)/2}$ is a square in $\F_p$ whenever $p\equiv 1,3\pmod 4$, i.e. always), which completes the proof.
\end{proof}

Proposition \ref{prop:legendre} guarantees that for $p>13$ a parameter $a$ satisfying the conditions of Proposition \ref{prop:mtr} can be found. The following technical lemma will be used to analyze explicit realizations for the cases $p\le 13$ not covered by Proposition \ref{prop:legendre}.

\begin{lem}\label{lem:realizations} Assume $n\ge p>2,\,p\nmid n$. Let $f\in\F_p[T][X]$ be a monic polynomial (in the variable $X$) and assume that $\Disc(f)=cT^m$, where $c\in\F_p^\times$ and $m\le p-1$. Let $K$ be the splitting field of $f$ over $\F_p(t)$ and $G=\Gal(K/\F_p(T))$.
Then
\begin{enumerate}\item[(i)] $K/\F_p(T)$ is unramified outside $0,\infty$ and is tamely 
ramified over 0.
\item[(ii)] If $m$ is odd and $G=S_n$ then $K/\F_p(T)$ is an $\lgtp$-realization of $S_n$.
\item[(iii)] If $n\ge 5$ and $G=A_n$ then $K/\F_p(T)$ is an $\lgtp$-realization of $A_n$.
\end{enumerate}
\end{lem} 

\begin{proof}{\bf (i).} Let $L=\F_p(T,\alpha)$, where $\alpha\in\overline{\F_p(T)}$ is a root of $f$. The extension $K$ is the Galois closure of $L$ over $\F_p(T)$, hence it is ramified over the same prime divisors. A finite prime $P$ of $\F_p(T)$ can ramify in $L$ only if $P\mid\Disc(f)$ and moreover if $e_P$ is its ramification index then $P^{e_P-1}\mid\Disc(f)$ and $P^{e_P}\mid\Disc(f)$ if $p|e_P$ (see e.g. \cite{Ros02}*{Corollary 7.10.2}), hence the extension $L/\F_p(T)$ (and therefore $K/\F_p(T)$) is unramified outside of $0,\infty$ and by the assumption $m\le p-1$ we have $e_T<p$ and the extension is tamely ramified over 0.

{\bf(ii).} It remains to show that $K/\F_p(T)$ is geometric. Since the only cyclic quotient of $S_n$ is $C_2$ (with kernel $A_n$), the only possible nontrivial extension of the field of constants of $\F_p(T)$ in $K$ is $\F_{p^2}(T)$, in which case we would have $\Gal(f/\F_{p^2}(T))=\Gal(K/\F_{p^2}(T))=A_n$. But then by the discriminant criterion $\Disc(f)=cT^m$ would be a square in $\F_{p^2}$, contradicting the assumption that $m$ is odd.

{\bf(iii).} In this case $G$ has no cyclic quotients, hence $K/\F_p(T)$ is geometric.
\end{proof}

The following proposition gives an alternative proof of Theorem \ref{thm:main} for the case $n=p+1$, assuming additionally $\F_q\not\supset\F_{p^2}$ if $G=S_{p+1}$.

\begin{prop}\label{prop:np1}Assume $p>2$, $q$ a power of $p$.
\begin{enumerate}\item[(i)] The group $S_{p+1}$ is $L(q)$-realizable if $\F_q\not\supset\F_{p^2}$.
\item[(ii)] The group $A_{p+1}$ is both $\lgtq$-realizable and $L(q)$-realizable.
\end{enumerate}
\end{prop}

\begin{proof} First we treat the case $q=p$. If $p>13$, by Proposition \ref{prop:legendre} one can find $2\le a\le\frac{p-1}2$ such that $(a,p+1)=1$ and $\LegendreP {a(a-1)}p$ takes any specified value $\pm 1$. Proposition \ref{prop:mtr} then gives the requisite $L(p)$-realizations of $S_{p+1},A_{p+1}$. Since $A_{p+1}$ is simple and has no cyclic quotients, any $A_{p+1}$-extension of $\F_p(T)$ is geometric and therefore the above $L(p)$-realizations of $A_{p+1}$ are $\lgtp$-realizations. 

Similarly, taking $(p,a)=(7,3),(11,5),(13,3)$ in Proposition \ref{prop:mtr} one obtains the requisite realizations for the pairs $(G,p)=(S_8,7),(A_{12},11),(S_{14},13)$. For the remaining $p\le 13$ cases, $\lgtp$-realizations of $S_4,S_6,A_8,S_{12},S_{14}$ respectively (as splitting fields of the polynomial $f$) are listed in the table below. They were found using a quick computer search. 
\\
\quad
\\
\begin{tabular}{|l|l|l|l|l|l|}
\hline
$p$&$f$&$\Disc(f)$&$\Gal(f/\F_p(T))$\\
\hline$3$ & $X^4 + X + T$&$T^3$& $S_4$\\
\hline$5$ & $X^6 + X^5 + 3X^3 + TX+T$ & $4T^3$ & $S_6$\\
\hline$7$ & $X^8 + 3X^2 + TX - 2$ & $4T^2$ & $A_8$\\
\hline$11$&$X^{12} + X^3 + 3X^2 + (T + 4)X + 1$&$2T^3$ &$S_{12}$\\
\hline$13$&$X^{14} + 7X^3 + 10X^2 + (9T^2 + 7)X + 1$&$-T^6$  & $A_{14}$\\
\hline
\end{tabular}
\\
\quad
\\
Here the discriminants and Galois groups were computed using the Magma Calculator\footnote{\url{http://magma.maths.usyd.edu.au/calc/}}. By Lemma \ref{lem:realizations} these polynomials indeed define $\lgtp$-realizations, except possibly for $f=X^4 + X + T$ which defines an $\lgtp$-realization by Proposition \ref{prop:lgtq}(ii). 
Finally, by Lemma \ref{lem:lgtq} the group $A_6=p(S_6)$ is also $\lgtp$-realizable and by Proposition \ref{prop:lgtq_lq} (and the observation above that for $p>3$ any $L(p)$-realization of $A_{p+1}$ is an $\lgtp$-realization since $A_{p+1}$ has no cyclic quotients) the groups $S_{p+1}$ with $p\le 13$ are $L(p)$-realizable and the groups $A_{p+1}$ with $p\le 13$ are $\lgtp$-realizable and also $L(p)$-realizable. 

We have shown that $A_{p+1},S_{p+1}$ are $L(p)$-realizable for all $p>2$. Finally we observe that since $S_{p+1}$ has no odd cyclic quotients, an $L(p)$-realization $K/\F_p(T)$ is linearly disjoint from $\F_q(T)$ if $\F_q\not\supset\F_{p^2}$ and $K\F_q/\F_q(T)$ is an $L(q)$-realization. Similarly, since $A_{p+1},p>3$ has no cyclic quotients any $L(p)$-realization $K/\F_p(T)$ is geometric (hence an $\lgtp$-realization) and then $K\F_q/\F_q(T)$ is an $\lgtq$-realization. In the case $p=3$ the same argument works starting with an $\lgtp$-realization (which by definition is geometric), which we have shown above to exist. This concludes the proof.

\end{proof}

\section{Legendre symbols with a GCD condition: proof of Proposition \ref{prop:legendre}}\label{sec:legendre}

Throughout the section $p> 13$ is a prime, $\sigma=\pm 1$. We want to show that for all such $p$ and $\sigma$ there exists an integer $2\le a\le\frac{p-1}2$ such that $\LegendreP {a(a-1)}p=\sigma$ and $(a,p+1)=1$. This is not hard to do for $p$ large enough using the P\'olya-Vinogradov method and available bounds on $n/\varphi(n)$ ($\varphi$ denotes the Euler totient function), however proving the assertion for all $p$ requires a more delicate sieve argument and a computer search for $p<7\cdot 10^7$ (the purpose of the sieve argument is to reduce the search to a manageable range).
The present section is purely number-theoretic and independent of the rest of the paper.

\subsection{Basic sieve}

\begin{observation}
To prove the assertion of Proposition \ref{prop:legendre} it suffices to find $1\le a<p$ such that $(a,p+1)$ and $\LegendreP{a(a-1)}{p}=\sigma$.
\end{observation}

\begin{proof}
Indeed if one finds $a$ as above but it falls in the range $\frac{p+1}2<a<p$ (note that $a\neq\frac{p+1}2$ because $(a,p+1)=1$) then $b=p+1-a$ satisfies the assertion of Proposition \ref{prop:legendre}.
\end{proof}

Now we turn to prove the existence of some $1<a<p$ with $\LegendreP{a(a-1)}{p}=\sigma$, and $(a,p+1)=1$. In what follows we fix $\sigma=\pm 1$.
We note that by inclusion-exclusion
\begin{equation}\label{eq:count_nd}
\#\left\{a\in\field{p}:\LegendreP{a(a-1)}{p}=\sigma, (a,p+1)=1\right\}=\sum_{d|p+1}\mu\left(d\right)N(d),
\end{equation}
where $\mu$ denotes the M\"obius function and 
\begin{multline}\label{eq:nd} N(d)=\#\left\{1<a<p:d|a,\LegendreP{a(a-1)}{p}=\sigma\right\}=\frac{1}{2}\sum_{1<a<p\atop{d|a}}\left(1+\sigma\cdot\LegendreP{a(a-1)}{p}\right)=\\=\frac{1}{2}\intval{\frac{p-1}{d}}+\frac{1}{2}\sum_{m=0}^{\intval{p/d}}{\eta_d}(m).\end{multline}
Here for $d|p+1$ (and hence $d$ coprime with $p$) and $m\in\Z$ we denote $${\eta_d}(m)=\sigma\cdot\LegendreP{dm(dm-1)}{p}=\sigma\cdot\LegendreP{m(m-1/d)}{p}.$$

\begin{claim}
We have $\abs{\sum_{m=0}^{\intval{p/d}}{\eta_d}(m)}\le 2\sqrt{p}(\log p+1)-2$.
\end{claim}

\begin{proof}

We consider the discrete Fourier transform modulo $p$ of ${\eta_d}(m)$ and the characteristic function of the interval $\mathds 1_{[0,H]}, H=\intval{\frac pd}$:
$$\fourierco{{\eta_d}}(k)=\frac{1}{p}\sum_{m=0}^{p-1}{\eta_d}(m)e^{-2\pi ikm/p},\quad 
\fourierco{\mathds 1}_{[0,H]}(k)=\frac 1p\sum_{m=0}^He^{-2\pi ikm/p}=\frac{e^{-2\pi i k(H+1)/p}-1}{p
\left(e^{-2\pi i k/p}-1\right)},\quad -\frac{p-1}2< k\le\frac{p-1}2.$$
By the Weil bound (and an elementary bound for $k=0$) we have  $$\abs{\fourierco{{\eta_d}}(k)}\le \frac{2}{\sqrt{p}}\quad (k\neq 0),\quad\abs{\fourierco{{\eta_d}}(0)}\le \frac 1p$$ (apply \cite{CaMo00}*{Main Theorem} with $f=x,g=x(x-1/d),\chi=\LegendreP{\cdot}{p},\psi=e^{2\pi i\cdot}$ in the notation of the cited theorem) and by elementary calculus
$$\left|\fourierco{\mathds 1}_{[0,H]}(k)\right|\le \frac 1{p\sin{\frac {\pi |k|}p}}\le\frac 1{|2k|}\quad \left(0<|k|\le\frac{p-1}2\right).$$ 
Hence by Parseval's identity,
$$
\abs{\sum_{m=0}^{\intval{\frac{p}{d}}}{\eta_d}(m)}=\abs{\sum_{m=0}^{p-1}\eta_d(m)\mathds 1_{[0,H]}(m)}=\abs{
p\sum_{k=-(p-3)/2}^{(p-1)/2}\fourierco{\eta_d}(k)\fourierco{\mathds 1}_{[0,H]}(k)}\le 1+2\sqrt p\sum_{k=1}^{(p-1)/2}\frac 1k\le 2\sqrt p(\log p+1)-2
$$
(for $p\ge 13$).
\end{proof}

By (\ref{eq:nd}) and the last claim, 
\begin{equation}\label{eq:nd_est}N(d)=\frac{p+1}{2d}+\xi(d),\end{equation} where $|\xi(d)|\le\sqrt{p}(\log p+1)$. Now by (\ref{eq:count_nd}),
\begin{equation}\label{eq:count1}
\#\left\{a\in\field{p}:\LegendreP{a(a-1)}{p}=\sigma,(a,p+1)=1\right\}=\frac{p+1}{2}\sum_{d|p+1}\frac{\mu(d)}d+\sum_{d|p+1}\mu(d)\xi(d)
=\frac{\varphi(p+1)}{2}+\sum_{d|p+1}\mu(d)\xi(d),
\end{equation}
where $\varphi$ is Euler's totient function. 

For a natural number $n$ we denote by $\omega(n)=\sum_{\ell|n\atop\mathrm{prime}}1$ its number of prime divisors and $W(n)=2^{\omega(n)}$. By (\ref{eq:count1}),
\begin{multline}
\label{eq:count2}
\#\left\{a\in\field{p}:\LegendreP{a(a-1)}{p}=\sigma, (a,p+1)=1\right\}
\ge
\frac{\varphi(p+1)}{2}-\sqrt{p}(\log p+1)\sum_{d|p+1}\mu(d)^2=\\=\frac{\varphi(p+1)}{2}-\sqrt{p}(\log p+1) W(p+1).
\end{multline}

%

\subsection{Proof for the case $\omega(p+1)\ge 13$}

In the present section we assume $\omega(p+1)\ge 13$ and prove Proposition \ref{prop:legendre} for this case using the estimates from the previous subsection. By (\ref{eq:count2}) the assertion follows at once from

\begin{lem}Let $N$ be a natural number with $\omega(N)\ge 13$. Then $\varphi(N)>2\sqrt N(\log N+1) W(N)$.\end{lem}

\begin{proof} Denote by $r_1<r_2<\ldots$ the sequence of (all) primes. Note that by assumption $$N\ge N_0=\prod_{i=1}^{13} r_i=304250263527210.$$ We note that for $N\ge N_0$ we have $N^{1/6}>2(\log N+1)$ (indeed the function $x^{1/6}/(\log x+1)$ is increasing for $x>e^5$ and $N_0^{1/6}> 2(\log N_0+1)$), hence it is enough to show $\varphi(N)\ge N^{2/3} W(N)$ for $N$ with $\omega(N)\ge 13$ (which automatically implies $N\ge N_0$).

Now if $\varphi(M)\ge M^{2/3} W(M)$ for some $M$ and $q|M$ is a prime then $ W(M)= W(qM)$ and 
$$\varphi(qM)=q\varphi(M)> q^{2/3}M^{2/3} W(M)=(qM)^{2/3} W(qM),$$ so it is enough to prove the assertion for squarefree $N$.

If $N=q_1\cdots q_k,k\ge 13$ is squarefree ($q_i$ prime) then
\begin{equation}\label{eq:phi_om_ratio}\frac{\varphi(N)}{N^{2/3} W(N)}=\prod_{i=1}^k\frac{q_i-1}{2q_i^{2/3}},\end{equation}
the function $(x-1)/2x^{2/3}$ is increasing for $x>1$ and $>1$ at $x=11$, so for given $k$ the ratio $\varphi(N)/N^{2/3} W(N)$ is minimized when $N=r_1\cdots r_k$ and since for $k=13$ the ratio is $>1$, by (\ref{eq:phi_om_ratio}) this remains the case for all $k\ge 13$, which concludes the proof.

\end{proof}

Now (\ref{eq:count2}) combined with the last lemma concludes the proof of Proposition \ref{prop:legendre} in the case $\omega(p+1)\ge 13$.

\subsection{Refining the sieve}

To handle the case $\omega(p+1)<13$ we need to refine our sieve. The main idea is to reduce the error term coming from the $\xi(d)$ by reducing the number of $d$ participating in the sieve. The price for this is a less tight main term, but it will improve the overall estimate with a correct choice of parameters, provided $\omega(p)<13$ and $p>7\cdot 10^7$. 

For $k|p+1$ denote (the second equality is by inclusion-exclusion),
$$
F(k)=\#\left\{a\in\field{p}:\LegendreP{a(a-1)}{p}=\sigma, (a,k)=1\right\}=\sum_{d|k}\mu(d)N(d).
$$
Now assume that $k|p+1$ is squarefree and write $$k=q_1\cdots q_l,\quad\rad(p+1)=k\cdot p_1\cdots p_s,$$ where $q_1<\ldots <q_l, \,p_1<\ldots<p_s$ are distinct primes and $\rad(n)$ denotes the radical of $n$ (the product of distinct primes dividing $n$). We have
$$
F(p+1)\ge F(k)-\sum_{i=1}^s (F(k)-F(kp_i))=\sum_{i=1}^s F(kp_i)-(s-1)F(k).
$$
and by (\ref{eq:nd_est}),
\begin{align*}
F(kp_i)=\sum_{d|kp_i}\mu(d)N(d)=\sum_{d|k}\mu(d)(N(d)-N(dp_i))=\sum_{d|k}\frac{\mu(d)(p+1)}{2d}
\left(
1-\frac 1{p_i}
\right)
+\sum_{d|k}\mu(d)(\xi(d)-\xi(dp_i)).
\end{align*}
and similarly
$$F(k)=\sum_{d|k}\frac{\mu(d)(p+1)}{2d}+\sum_{d|k}\mu(d)\xi(d).$$
Using the estimate $\xi(d)\le\sqrt p(\log p+1)$,
\begin{align*}
\abs{
F(kp_i)-
\left(
1-\frac{1}{p_i}
\right)
F(k)
}
\le \sum_{d|k}\mu^2(d)\left(1+\frac 1p_i\right)\sqrt p(\log p+1)
= (1+1/p_i) W(k)\sqrt{p}(\log p+1) 
\end{align*}
Thus, we get 
\begin{multline}\label{eq:fk_est1}
F(p+1)\ge\sum_{i=1}^s F(kp_i)-(s-1)F(k)
\ge \sum_{i=1}^s\left(F(kp_i)-\left(1-\frac{1}{p_i}\right)F(k)\right)+\left(1-\sum_{i=1}^s\frac{1}{p_i}\right)F(k)\ge\\
\ge \left(1-\sum_{i=1}^s\frac{1}{p_i}\right)F(k)-\left(s+\sum_{i=1}^s\frac 1{p_i}\right) W(k)\sqrt{p}(\log p+1).
\end{multline}

By a similar calculation to (\ref{eq:count2}) we have 
\begin{equation}\label{eq:fk_est2}
F(k)\ge \frac{p\varphi(k)}{2k}-\sqrt{p}(\log p+1)  W(k).
\end{equation}
Combining (\ref{eq:fk_est1}) and (\ref{eq:fk_est2}) and denoting $\delta=1-\sum_{i=1}^s\frac{1}{p_i}$ we get 
$$
F(p+1)\ge
\frac{\delta p\varphi(k)}{2k}-
(s+1)
\sqrt{p}(\log p+1) W(k)
$$
So, if for some $p\ge 13$ and squarefree $k|p+1$ we can show
$$
\frac{\delta p\varphi(k)}{2k}>
(s+1)
\sqrt{p}(\log p+1) W(k)=(s+1)
\sqrt{p}(\log p+1)2^l,
$$
where $\delta=1-\frac 1{p_1}-\ldots -\frac 1{p_s}$ and $p_1,\ldots,p_s$ are the distinct prime divisors of $p+1$ not dividing $k$,
the assertion of Proposition $\ref{prop:legendre}$ would follow for this $p$. The last inequality is equivalent to
\begin{equation}\label{eq:sieve_cond}\frac{\sqrt{p}}{\log p+1}>\frac{(s+1)2^{l+1}}{\left(1-\sum_{i=1}^s\frac 1{p_i}\right)\prod_{j=1}^l\left(1-\frac 1{q_j}\right)}.\end{equation}

\subsection{The case $9\le\omega(p+1)\le 12$}\label{sec:intermediate_range}

We keep the setting and notation of the previous subsection and denote by $r_1,r_2,\ldots$ the complete sequence of primes. We now make the additional assumption that $q_j<p_i$ for all $i,j$ (so $k$ is the product of the $l$ smallest prime divisors of $p+1$). We note that for fixed $s,l$ the right hand side of (\ref{eq:sieve_cond}) is decreasing in each $p_i,q_j$ (as long as $\delta>0$) and therefore the right hand side of (\ref{eq:sieve_cond}) is bounded from above by
$$R_{l,s}=\frac{(s+1)2^{l+1}}{\left(1-\sum_{i=l+1}^{l+s}\frac 1{r_i}\right)\prod_{j=1}^l\left(1-\frac 1{r_j}\right)},$$
as long as $R_{l,s}>0$. On the other hand the left hand side of (\ref{eq:sieve_cond}) is bounded from below by 
$$L_{l,s}=\frac{\sqrt{r_1\cdots r_{s+l}}}{\log\left(r_1\cdots r_{s+l}\right)+1}.$$ Whenever $L_{l,s}>R_{l,s}>0$ it is guaranteed that the inequality (\ref{eq:sieve_cond}) is satisfied whenever $\omega(p+1)=l+s$ and Proposition \ref{prop:legendre} holds for $p$.

In the table below for each value of $1\le n\le 12$ we list the choice of $l,s$ with $l+s=n$ minimizing $R_{l,s}>0$  and the corresponding values of $R_{l,s},L_{l,s}$ (rounded down to 3 decimal places).
\quad
\begin{center}
\begin{tabular}{|l|l|l|l|l|}
\hline
$n$ & $l$ & $s$ & $R_{l,s}$ & $L_{l,s}$\\
\hline
$1$ & $0$ & $1$ & $8$ & $0.835$\\
$2$ & $1$ & $1$ & $24$ & $0.877$\\
$3$ & $1$ & $2$ & $51.428$ & $1.244$\\
$4$ & $1$ & $3$ & $98.823$ & $2.283$\\
$5$ & $2$ & $3$ & $169.541$ & $5.495$\\
$6$ & $2$ & $4$ & $245.242$ & $15.322$\\
$7$ & $2$ & $5$ & $334.504$ & $50.519$\\
$8$ & $2$ & $6$ & $444.614$ & $182.262$\\
$9$ & $2$ & $7$ & $574.201$ & $738.575$\\
$10$ & $2$ & $8$ & $720.253$ & $3409.625$\\
$11$ & $2$ & $9$ & $896.738$ & $16571.694$\\
$12$ & $2$ & $10$ & $1097.213$ & $88920.402$\\
\hline
\end{tabular}
\end{center}
\quad

We see that for $9\le n\le 12$ there is a suitable choice of $l$ such that $L_{l,s}>R_{l,s}>0$ and by the observations above this concludes the proof for the case $9\le\omega(p+1)\le 12$.

\subsection{The remaining primes}

It remains to deal with the case $\omega(p+1)\leq 8$. Looking again at the table from the previous subsection we see that for $n\le 8$ one can choose suitable $l,s$ with $l+s=n$ such that $0<R_{l,s}<445$. By the same reasoning used in the previous subsection, the assertion of Proposition \ref{prop:legendre} holds for $p$ whenever $\omega(p+1)=n=l+s$ and $\frac{\sqrt p}{\log p+1}>R_{l,s}>0$.

Note that the function $\frac {\sqrt x}{\log x+1}$ is increasing for $x>e$ and for $x=7\cdot 10^7$ 
we have $\frac {\sqrt x}{\log x+1}=464.164\ldots>445$, hence counterexamples to Proposition 
\ref{prop:legendre} can only occur with $p<7\cdot 10^7$. By similar reasoning counterexamples with 
$\omega(p+1)\le 5$ can only occur with $p<y=9\cdot 10^6$, because $\frac{\sqrt y}{\log y+1}
=177.944\ldots$ exceeds all the values of $R_{l,s}$ corresponding to $n\le 5$ in the table above.

To conclude the proof for the range $p<7\cdot 10^7$ we have run a SageMath script that runs through all primes $13< p<7\cdot 10^7$ with $6\le\omega(p+1)\le 8$ and the primes $13<p<9\cdot 10^6$ with $\omega(p+1)\le 5$, and for each $p$ goes over $a=1,2,3,\ldots$ until it finds an $a$ with $\LegendreP {a(a-1)}p=1$ and $(a,p+1)=1$ and another one with $\LegendreP {a(a-1)}p=-1$ and $(a,p+1)=1$. For all primes in the range suitable values of $a$ were found, which concludes the proof of Proposition \ref{prop:legendre}. The runtime for the script was about 10 hours on a standard 4 GHz desktop PC. An implementation in a low-level programming language like C would surely run much faster, but we have not attempted it.

The SageMath scripts used for generating the table in the previous subsection and checking the range $13<p<7\cdot 10^7$ appear in the appendix below. 

\appendix 

\section{SageMath code}

\subsection{Code generating the table in subsection \ref{sec:intermediate_range}}

\begin{verbatim}

# Primes is the list of first 100 primes
# RR is the field of real numbers

Primes = list(primes(1000))[:100]
RR = RealField()

# the function R(l,s) computes the value of R_{l,s}

def R(l, s):
    qlist = Primes[:l]      
    plist = Primes[l:l+s]
    return RR((s+1)*2^(l+1)/(prod([1-1/q for q in qlist])*(1-sum([1/p for p in plist]))));
    
# for given n find l that (with s=n-l) minimizes R_{l,s} (if positive) and return l, R_{l,s}

def best_R(n):
    return min([(R(l, n-l),l) for l in range(n+1) if R(l, n-l)>0])

# the function L(l,s) computes the value of L_{l,s}

def L(l, s):
    P = product(Primes[:s+l])
    return RR(sqrt(P)/(log(P)+1))
    
# for each 1 <= n <= 11 find the l that minimizes R_{l,s} and print out table row    
    
for n in range(1,12):
    (R_value, l) = best_R(n)
    s = n - l
    print n, l, s, R_value, L(l, s)
\end{verbatim}

\subsection{Code checking the range $13<p<7\cdot 10^7$}

\begin{verbatim}

X = 7*10^7       # counterexamples with 6<=omega(p+1)<=8 must be < X
Y = 9*10^6       # counterexamples with omega(p+1)<6 must be < Y

Primes = list(primes(X))  #list of primes < X       

# computing the values of omega(m) for m <= X

omega = [0 for i in range(X+1)]
for p in Primes:
    for i in range(0,X+1,p):
        omega[i] += 1
        
# go over all primes p<X with 6<=omega(p+1)<=8 and p<Y with omega(p+1)<6 and 
# search exhaustively for suitable (odd) values of a.  
# if suitable values of a are not found (for sigma=1 or sigma=-1) then
# the counterexample p is printed
    
for p in Primes:
    if p<=13:
        continue
    if p>Y and omega[p+1]<6:
        continue
    if omega[p+1]>8:
        continue
    a_found_plus = false       # flag indicating suitable a with sigma = 1  has been found
    a_found_minus = false      # flag indicating suitable a with sigma = -1 has been found
    for a in range(3,p,2):     # going over odd 3 <= a <p
        if gcd(a,p+1) != 1:
            continue
        L = legendre_symbol(a*(a-1),p)
        if L == 1:
            a_found_plus = true
        else:
            a_found_minus = true
        if a_found_plus and a_found_minus:
            break
    if a_found_plus and a_found_minus:
        continue
    print 'COUNTEREXAMPLE: ', p
    
\end{verbatim}

{\bf Acknowledgments.} The authors would like to thank the anonymous referee of a previous version of this paper for pointing out the polynomial (\ref{eq:new pol}) to us, which simplified and slightly improved our main results in the important case $n=p+1$. The authors would also like to thank Ze\'ev Rudnick for his remarks on a previous version of the paper. The present work used the SageMath and Magma Calculator platforms to perform calculations and computer searches necessary for some of the proofs.
Both authors were partially supported by a grant of the Israel Science Foundation no. 2507/19. 

\bibliography{../../Tex/mybib}
\bibliographystyle{amsrefs}

\end{document}